\documentclass[12pt,a4paper,oneside]{article}
\usepackage[utf8]{inputenc}
\usepackage{amsmath}
\usepackage{amsfonts}
\usepackage{amssymb}
\usepackage{graphicx}
\usepackage[english]{babel}
\usepackage{color}
\usepackage{tikz}
\usepackage[T1]{fontenc}
\usepackage{nicefrac}
\usepackage{hyperref}
\usepackage{cite}
\usepackage{dsfont}
\usepackage{setspace}

\newcommand{\re}{\mathbb{R}}
\newcommand{\na}{\mathbb{N}}

\newcommand{\eps}{\varepsilon}

\newcommand{\diver}{\operatorname{div}}

\newcommand{\intq}{\int_{Q_T}}
\newcommand{\supp}{\operatorname{supp}}
\newcommand{\uzmn}{u_0^{m,n}}
\newcommand{\fmn}{f^{m,n}}
\newcommand{\umn}{u_{m,n}}
\newcommand{\vmn}{v_{m,n}}
\newcommand{\vzmn}{v_0^{m,n}}

\newtheorem{defi}{Definition}[section]

\newtheorem{lem}[defi]{Lemma}
\newtheorem{teo}[defi]{Theorem}

\newtheorem{remark}[defi]{Remark}

\newenvironment{proof}{\noindent{\textit{Proof.}}}{\hfill$\square$}

\numberwithin{equation}{section}

\title{Entropy solutions for time-fractional porous medium type equations}
\date{}
\author{Kerstin Schmitz\footnote{Fakultät für Mathematik, Universität Duisburg-Essen, Thea-Leymann-Str.~9, 45127~Essen, Germany, E-mail: kerstin.schmitz.ks@outlook.de}\hspace{0.1cm} and Petra Wittbold\footnote{Fakultät für Mathematik, Universität Duisburg-Essen, Thea-Leymann-Str.~9, 45127~Essen, Germany, E-mail: petra.wittbold@uni-due.de}}

\begin{document}

\maketitle

\begin{abstract}
In this paper we prove existence of entropy solutions to the time-fractional porous medium type equation,
$$\partial_t[k\ast(u-u_0)]-\diver (A(t,x)\nabla\varphi(u))=f\text{ in }Q_T=(0,T)\times\Omega,$$
with Dirichlet boundary condition, initial condition $u(0,\cdot)=u_0$ in $\Omega$, and $L^1$-data $f\in L^1((0,T)\times\Omega), u_0\in L^1(\Omega)$. To this end we approximate the data by $L^\infty$-functions, use a known existence result of weak solutions for these more regular data, and additionally a known contraction principle for weak solutions, which can be adopted to the entropy solutions.\\
\end{abstract}



\section{Introduction}

We consider for $T>0$ and $\Omega\subseteq\re^d$ a bounded domain, $d\in\na$, the problem
\begin{align}\label{equation}\tag{$P(u_0,f)$}
\begin{cases}
\begin{aligned}
\partial_t[k\ast(u-u_0)]-\diver (A(t,x)\nabla\varphi(u))&=f &&\text{in }Q_T=(0,T)\times\Omega\\
u&=0 &&\text{in }\Sigma=(0,T)\times\partial\Omega\\
u(0,\cdot)&=u_0&&\text{in }\Omega,
\end{aligned}
\end{cases}
\end{align}
where for $t\in[0,T]$ we define
\begin{align*}
(k\ast v)(t):=\int_0^tk(t-\tau)v(\tau)\,d\tau.
\end{align*}

We make the following assumptions:
\begin{itemize}
\item[(H$k$)] $k\in L^1_\text{loc}(\re^+)$ is non-negative, non-increasing, and there exists $l\in L^p(0,T)$ with $p>1$ such that $k\ast l=1$ in $(0,\infty)$.
\item[(H$A$)] $A\in L^\infty((0,T)\times\Omega;\re^{d\times d})$ and there exists $\nu>0$ such that
\begin{align*}
(A(t,x)\xi,\xi)\geq\nu|\xi|^2\quad\forall\xi\in\re^d\text{ and  a.e. }(t,x)\in Q_T.
\end{align*}
\item[(H$\varphi$)] $\varphi\in C^1(\re)$, $\varphi'(r)\geq 0$ for all $r\in\re$, $\varphi(0)=0$, $\varphi$ is strictly increasing in $\re$, and there exist $\mu,R>0$ such that
\begin{align*}
0<\mu\leq \varphi'(r)\quad\forall r\in\re\text{ with }|r|>R.
\end{align*}
\item[(Hd)] $u_0\in L^1(\Omega)$, $f\in L^1(Q_T)$.
\end{itemize}

Note that kernels $k$ satisfying (H$k$) are in particular kernels of type $\mathcal{P}\mathcal{C}$, which have been studied by many authors, see, e.g., \cite{VZ15,KSVZ16,K11,VZ17}. Kernels of type $\mathcal{P}\mathcal{C}$ are used in applications to model subdiffusion processes. Subdiffusion is a special case of anomalous diffusive behaviour which is in the force free-limit slower than Brownian motion. For more informations, see \cite{MK00,MK04}.
An important example is given by $(k,l):=(g_{1-\alpha},g_\alpha)$ for $\alpha\in(0,1)$, where 
\begin{align*}
g_\beta(t):=\frac{t^{\beta-1}}{\Gamma(\beta)}\text{ for }t>0,\ \beta>0.
\end{align*}
In this case $\partial_t(k\ast v)$ represents the Riemann-Liouville fractional derivative of order $\alpha$ and $k\ast\partial_t v$ the Caputo derivative if $v$ is sufficient smooth. Note, that in this case the condition $l\in L^p(0,T)$ is satisfied.
Some further examples of kernels satisfying (H$k$) are the time-fractional case with exponential weight:
$$k(t):=e^{-\gamma t}g_{1-\alpha}(t),\ l(t):=e^{-\gamma t}g_\alpha(t)+\gamma [1\ast (g_\alpha e^{-\gamma \cdot})](t)$$
for $\alpha\in(0,1)$ and $\gamma >0$,
and the ultra-slow diffusion case: $$k(t):=\int_0^1 g_\beta(t)\,d\beta, \ l(t):=\int_0^\infty\frac{e^{-st}}{1+s}\,ds,$$ which is considered in \cite{K08, K11, KR18, KSVZ16}.\\

By assuming (H$\varphi$) we cover degenerated time-fractional equations. For example we can chose $\varphi(r):=|r|^{m-1}r$ for $m>1$, so that \ref{equation} becomes a porous medium equation, which has been studied in \cite{A19, ACV17, DVV19, DNA19, P15, VZ15}.\\

In applications \ref{equation} appears in the modelling of dynamic processes with memory, for example, to model heat conduction with memory (see \cite{N71,Pruess}) and diffusion in fluids in porous media with memory \cite{C99,JDiss}.\\

Existence of weak solutions to \ref{equation} and additionally a contraction principle for weak solutions were shown for more regular data $u_0,f$ in \cite{WWZ20}. In the linear case existence and uniqueness of weak solutions were shown in \cite{Z19, LRS19, VZ17}. For the porous medium operator $L^1$ is a natural space guaranteeing the monotonicity property and also from the physical point of view $L^1$ is a useful space for several evolution problems, e.g., the transport of fluids in porous media, and heat conduction. In the setting of $L^1$-data we cannot expect weak solutions. Therefore, we work with entropy solutions.\\

For the doubly-nonlinear history-dependent (degenerated) problem with a time-independent operator existence and uniqueness of entropy solutions (also in the case of $L^1$-data) were shown in \cite{SW18,JW04,S20}. Here the theory about generalized solutions for integro-differential equations (see \cite{G85}), using the $m$-accretivity of the time-independent operator, is applied. In the case of a time-dependent operator we cannot apply this approach. Note, that even in the linear case, i.e., $\varphi= id$ no existence results for $L^1$-data are known.\\

Note that there are several articles dealing with decay estimates for time-fractional (porous medium type) equations, see, e.g., \cite{KSVZ16,DVV19,VZ15}.\\

The paper is structured in the following way: In Section \ref{sec_L^infty_data} we consider bounded data $u_0\in L^\infty(\Omega),f\in L^\infty(Q_T)$. In this case, existence of weak solutions was shown in \cite{WWZ20}. We prove, that a weak solution to \ref{equation} is also an entropy solution to \ref{equation} by using the fundamental identity (see \cite[Lemma 2.1]{WWZ20}).\\
Afterwards, we formulate in Section \ref{sec_contraction_principle} a contraction principle for the weak solutions, which is a technical extension of the contraction principle formulated in \cite{WWZ20}.\\
In Section \ref{sec_approximation} we consider general data $u_0\in L^1(\Omega),f\in L^1(Q_T)$ and approximate them by functions $u_0^{m,n}\in L^\infty(\Omega), f^{m,n}\in L^\infty(Q_T)$. We know, that there exists an entropy solution to the approximated equation $P(u_0^{m,n},f_{m,n})$ and can show by the contraction principle that $u_{m,n}$ converges to a function $u\in L^1(Q_T)$.\\
In Section \ref{sec_passage_to_the_limit} we then pass to the limit in the equation. Here, we use the coercivity condition (H$A$) of the operator $A$, and, furthermore, the fact that $\varphi$ is increasing to take advantage of the monotone convergences of the approximations.

\section{Entropy solutions in the case of $L^\infty$-data}\label{sec_L^infty_data}

The idea is to approximate the data $u_0,f$ by bounded data in $L^\infty(\Omega)$, $L^\infty(Q_T)$ respectively. By \cite[Theorem 6.1]{WWZ20}, we know that \ref{equation} then admits a weak solution. We first show that any weak solution to \ref{equation} is an entropy solution.\\
For a space $V\subseteq W^{1,1}(0,T;X)$, where $X$ is a Banach space, we denote by ${}_0V$ the space of all $\psi\in V$ that vanish at $t=0$.
We set
\begin{align*}
W_\varphi(T,u_0):=\big\{w\in L^2(0,T;L^2(\Omega)):&\ k\ast(w-u_0)\in {}_0W^{1,1}(0,T;H^{-1}(\Omega))\\
&\text{ and }\varphi(w)\in L^2(0,T;H^1_0(\Omega)) \big\}.
\end{align*}

\begin{defi}
Let (H$k$), (H$A$), (H$\varphi$), and (Hd) be satisfied. A function $u\in W_\varphi(T,u_0)$ is a weak solution to \ref{equation}, if for any test function $\eta\in W^{1,1}(0,T;L^2(\Omega))\cap L^2(0,T;H^1_0(\Omega))$ with $\eta(T,\cdot)=0$ there holds
\begin{align*}
\int_0^T\int_\Omega-\eta_t[k\ast(u-u_0)]+(A\nabla\varphi(u),\nabla\eta)\,dx\,dt=\int_0^T\int_\Omega f\eta\,dx\,dt.
\end{align*}
\end{defi}

Under the regularity condition
\begin{align}\label{regularity_assumption}
k\ast(u-u_0)\in{}_0W^{1,1}(0,T;L^1(\Omega)),
\end{align}
one can show by an approximation argument that for a weak solution $u$
\begin{align}\label{eq_weaksolution}
\int_0^T\int_\Omega\eta\partial_t[k\ast(u-u_0)]+(A\nabla\varphi(u),\nabla\eta)\,dx\,dt=\int_0^T\int_\Omega f\eta\,dx\,dt
\end{align}
is satisfied for all $\eta\in L^2(0,T;H^1_0(\Omega))\cap L^\infty(Q_T)$, and by a cut-off function argument this is equivalent to
\begin{align}\label{eq_weaksolution_t1}
\int_0^{t_1}\int_\Omega\eta\partial_t[k\ast(u-u_0)]+(A\nabla\varphi(u),\nabla\eta)\,dx\,dt=\int_0^{t_1}\int_\Omega f\eta\,dx\,dt
\end{align}
for all $t_1\in(0,T]$ and all $\eta\in L^2(0,T;H^1_0(\Omega))\cap L^\infty(Q_T)$.

Since $\varphi\in C^1(\re)$ is strictly increasing and $\varphi(0)=0$, we can define the function $b:=\varphi^{-1}$, which is continuous, strictly increasing, and satisfies $b(0)=0$. If we define $v:=\varphi(u)$ and $v_0:=\varphi(u_0)$, then \ref{equation} is equivalent to
\begin{align}
\begin{cases}
\begin{aligned}
\partial_t[k\ast(b(v)-b(v_0))]-\diver (A(t,x)\nabla v)&=f &&\text{in }Q_T\\
v&=0 &&\text{in }\Sigma\\
v(0,\cdot)&=v_0&&\text{in }\Omega.
\end{aligned}
\end{cases}
\end{align}

We define an entropy solution to \ref{equation} based on the definition in \cite{JW04}. Therefore, we set
\begin{align*}
\mathcal{P}:=\left\{S\in C^1(\re): 0\leq S'\leq 1,\ \supp S' \text{ compact},\ S(0)=0 \right\}.
\end{align*}

\begin{defi}
Let (H$k$), (H$A$), (H$\varphi$) and (Hd) be satisfied. A measurable function $v:Q_T\rightarrow\re$ is called an entropy solution to \ref{equation} if $b(v)\in L^1(Q_T), T_K(v)\in L^2(0,T;H^1_0(\Omega))$ for all $K>0$, and
\begin{align*}
&-\int_{Q_T}\zeta_t\left[k_1\ast\int_{v_0}^v S(\sigma-\phi)\,db(\sigma) \right]+\int_{Q_T}\zeta\partial_t[k_2\ast(b(v)-b(v_0))]S(v-\phi)\\
&+\intq\zeta (A(t,x)\nabla v,\nabla S(v-\phi))\leq\intq\zeta f S(v-\phi)
\end{align*}
for all $\phi\in H^1_0(\Omega)\cap L^\infty(\Omega)$, $\zeta\in\mathcal{D}([0,T)),\zeta\geq0$, $S\in\mathcal{P}$, and $k_1,k_2\in L^1(0,T)$ non-increasing and non-negative with $k=k_1+k_2$ and $k_2(0^+)<\infty$.
\end{defi}

In order to show that a weak solution $u$ to \ref{equation} is also an entropy solution to \ref{equation}, we will use $S(v-\phi)\zeta$ as a test function where $S\in\mathcal{P},\phi\in H^1_0(\Omega)\cap L^\infty(\Omega),$ and $\zeta\in\mathcal{D}([0,T))$ with $\zeta\geq 0$. Since $S(v-\phi)\zeta$ is an element of $L^2(0,T;H^1_0(\Omega))\cap L^\infty(Q_T)$, but not in $W^{1,1}(0,T;L^2(\Omega))\cap L^2(0,T;H^1_0(\Omega))$, we have to assume \eqref{regularity_assumption}.

\begin{lem}\label{lem_weak_entropy_solution}
Let (H$k$), (H$A$), (H$\varphi$), and (Hd) be satisfied. If $u$ is a weak solution to \ref{equation}, which satisfies \eqref{regularity_assumption}, then $v=\varphi(u)$ is an entropy solution to \ref{equation}.
\end{lem}

\begin{proof}
Let $u$ be a weak solution to \ref{equation} and $S\in\mathcal{P},\phi\in H^1_0(\Omega)\cap L^\infty(\Omega), \zeta\in\mathcal{D}([0,T)),$ and $v:=\varphi(u), v_0:=\varphi(u_0)$. We use $S(v-\varphi)\zeta$ as a test function in \eqref{eq_weaksolution} to obtain
\begin{align*}
\intq\zeta S(v-\phi)\partial_t[k\ast(b(v)-b(v_0))]+\intq \zeta(A\nabla v,\nabla S(v-\phi))=\intq fS(v-\phi)\zeta.
\end{align*}
Now we choose arbitrary $k_1,k_2\in L^1(0,T)$ non-increasing and non-negative with $k_2(0^+)<\infty$ such that $k=k_1+k_2$. For $\lambda>0$ we define $k_{1,\lambda}$ by the kernel associated to the Yosida-approximation of the operator $L_1:=\partial_t(k_1\ast\cdot)$, $D(L_1):=\{w\in L^1(0,T;L^1(\Omega)):k_1\ast w\in {}_0W^{1,1}(0,T;L^1(\Omega))\}$. Then there holds
\begin{align*}
&\intq\zeta S(v-\phi)\partial_t[k_{1,\lambda}\ast(b(v)-b(v_0))]+\intq\partial_t[k_2\ast(b(v)-b(v_0))]S(v-\phi)\zeta\\
&+\intq\zeta(A\nabla v,\nabla S(v-\phi))\\
&=\intq fS(v-\phi)\zeta+\intq\zeta S(v-\phi)\partial_t[(k_{1,\lambda}-k_1)\ast(b(v)-b(v_0))].
\end{align*}
Using the fundamental identity (see \cite[Lemma 2.1]{WWZ20}) for the first term in the above equation we get
\begin{align*}
&-\intq\zeta_t\left(k_{1,\lambda}\ast\int_{v_0}^vS(\sigma-\phi)\,db(\sigma)\right)+\intq\zeta S(v-\phi)\partial_t[k_2\ast(b(v)-b(v_0))]\\
&+\intq\zeta (A\nabla v,\nabla S(v-\phi))\\
&\leq\intq\zeta fS(v-\phi)+\intq\zeta S(v-\phi)\partial_t[(k_{1,\lambda}-k_1)\ast(b(v)-b(v_0))].
\end{align*}
Since $k_{1,\lambda}\rightarrow k_1$ in $L^1(0,T)$ and $b(v)-b(v_0)\in D(L_1)$ and so $\partial_t[k_{1,\lambda}\ast(b(v)-b(v_0))]\rightarrow \partial_t[k_1\ast(b(v)-b(v_0))]$ in $L^1(Q_T)$, we obtain, by passing to the limit in the above equation that $v$ is an entropy solution to \ref{equation}.
\end{proof}

\section{Contraction principle}\label{sec_contraction_principle}

\begin{lem}\label{Contraction principle}
Let $(k,l)\in\mathcal{P}\mathcal{C}$, (H$A$) be satisfied and $\varphi\in C^1(\re)$ a strictly increasing function in $\re$. Let $u_i\in W_\varphi(u_{0,i},f_i), i=1,2$, be weak solutions to problem \ref{equation} with $u_0=u_{0,i}\in L^1(\Omega)$ and $f=f_i\in L^1(Q_T)$ in the sense that \eqref{regularity_assumption} is fulfilled and in particular \eqref{eq_weaksolution_t1} holds true for all $\eta\in L^2(0,T;H^1_0(\Omega))\cap L^\infty(Q_T)$. Then
\begin{align}
\|u_1-u_2\|_{L^1(Q_T)}&\leq T\|u_{0,1}-u_{0,2}\|_{L^1(\Omega)}+\|l\|_{L^1(0,T)}\|f_1-f_2\|_{L^1(Q_T)} \label{contraction_principle_betrag}\\
\intq(u_1-u_2)^+&\leq T\int_\Omega(u_{0,1}-u_{0,2})^++\|l\|_{L^1(0,T)}\intq(f_1-f_2)^+ \label{contraction_principle_+}\\
\intq(u_1-u_2)^-&\leq T\int_\Omega(u_{0,1}-u_{0,2})^-+\|l\|_{L^1(0,T)}\intq(f_1-f_2)^-. \label{contraction_principle_-}
\end{align}
\end{lem}

\begin{proof}
Inequality \eqref{contraction_principle_betrag} was shown in \cite[Theorem 7.1]{WWZ20}. The proof of \eqref{contraction_principle_+} and \eqref{contraction_principle_-} is analogous to that of Theorem 7.1 in \cite{WWZ20} with the only difference that in the case of \eqref{contraction_principle_+}, \eqref{contraction_principle_-} we approximate $\re\ni y\mapsto y^+$ and $\re\ni y\mapsto y^-$ by $H_\eps(y):=\sqrt{(y^+)^2+\eps^2}-\eps$, $H_\eps(y):=\sqrt{(y^-)^2+\eps^2}-\eps$, respectively, for $\eps>0$.
\end{proof}

\section{Approximation}\label{section_approximation}\label{sec_approximation}

Let (H$k$),(H$A$),(H$\varphi$), and (Hd) be satisfied. For $m,n\in\na$ we define 
\begin{align*}
\uzmn:=
\begin{cases}
m, &\text{if }u_0> m\\
u_0, &\text{if }-n\leq u_0\leq m\\
-n, &\text{if }u_0<-n
\end{cases}
\text{ and }
\fmn:=
\begin{cases}
m,&\text{if }f>m\\
f,&\text{if }-n\leq f\leq m\\
-n, &\text{if }f<-n.
\end{cases}
\end{align*}
By \cite[Theorem 6.1]{WWZ20} $P(\uzmn,\fmn)$ admits a weak solution $\umn\in W_\varphi(T,\uzmn)\cap L^\infty(Q_T)$.

\begin{lem}\label{lem_a.e.convergence_umn}
Let $\umn\in W_\varphi(T,\uzmn)\cap L^\infty(Q_T)$ be a weak solution to $P(\uzmn,\fmn)$ for any $m,n\in\na$. For fixed $n\in\na$, there exists an element $u_{\infty,n}\in L^1(Q_T)$ such that
$$\umn\rightarrow u_{\infty,n}\text{ a.e. in }Q_T\text{ for }m\rightarrow\infty.$$
Moreover there exists a function $u\in L^1(Q_T)$ such that
$$u_{\infty,n}\rightarrow u\text{ a.e. in }L^1(Q_T)\text{ for }n\rightarrow\infty.$$
\end{lem}

\begin{proof}
Using \eqref{contraction_principle_+} and \eqref{contraction_principle_-} we know that for all $m,n\in\na$
\begin{align}\label{umn_increasing_decreasing}
\umn\leq u_{m+1,n}\quad\text{and}\quad \umn\geq u_{m,n+1}.
\end{align}
From Lemma \eqref{contraction_principle_betrag} we further obtain
\begin{align*}
\sup_{m,n\in\na}\|\umn\|_{L^1(Q_T)}&\leq\sup_{m,n\in\na}(T\|\uzmn\|_{L^1(\Omega)}+\|l\|_{L^1(0,T)}\|\fmn\|_{L^1(Q_T)})\\
&\leq T\|u_0\|_{L^1(\Omega)}+\|l\|_{L^1(0,T)}\|f\|_{L^1(Q_T)}.
\end{align*}
As a consequence we know that the increasing sequence $(\umn)_{m\in\na}$, for fixed $n\in\na$, converges a.e. in $Q_T$ towards an element $u_{\infty,n}$ for $m\rightarrow\infty$.
From \eqref{umn_increasing_decreasing} it follows $u_{\infty,n}\geq u_{\infty,n+1}$ for all $n\in\na$ and therefore we obtain by the same argumentation, that $u_{\infty,n}$ converges a.e. in $Q_T$ for $n\rightarrow\infty$ towards an element $u$. Using \eqref{contraction_principle_betrag} and Fatou's Lemma, we get for any $n\in\na$
\begin{align*}
\begin{split}
\intq|u_{\infty,n}|&\leq\liminf_{m\rightarrow\infty}\intq|\umn|\\
&\leq\liminf_{m\rightarrow\infty}(T\|\uzmn\|_{L^1(\Omega)}+\|l\|_{L^1(0,T)}\|\fmn\|_{L^1(Q_T)})\\
&\leq T\|u_0\|_{L^1(\Omega)}+\|l\|_{L^1(0,T)}\|f\|_{L^1(Q_T)}.
\end{split}
\end{align*}
Consequently,
\begin{align}\label{u_in_L1}
\|u\|_{L^1(Q_T)}&\leq\liminf_{n\rightarrow\infty}\intq|u_{\infty,n}|\leq T\|u_0\|_{L^1(\Omega)}+\|l\|_{L^1(0,T)}\|f\|_{L^1(Q_T)},
\end{align}
which implies $u_{\infty,n}\in L^1(Q_T)$ and $u\in L^1(Q_T)$.
\end{proof}

\begin{lem}\label{lem_bound_b(vmn)}
There exists, for any $n\in\na$, a function $g^n\in L^1(Q_T)$ and, moreover, there exists a function $g\in L^1(Q_T)$ which is independent of $n$, such that a.e. in $Q_T$
\begin{align*}
|\umn|\leq g^n\ \  \forall m,n\in\na\quad\text{ and }\quad|u_{\infty,n}|\leq g\ \  \forall n\in\na.
\end{align*}
\end{lem}

\begin{proof}
First let $n\in\na$ be fixed. Since $(\umn)_{m\in\na}$ is an increasing function which converges to $u_{\infty,n}$, we know that $\umn\leq u_{\infty,n}$ for all $m\in\na$ and in particular $(\umn)^+\leq (u_{\infty,n})^+$ for all $m\in\na$. Additionally
$$(\umn)^-=\max\{0,-\umn\}\leq\max\{0,-u_{n,1}\}=(u_{n,1})^-.$$
Consequently,
$$|\umn|=(\umn)^++(\umn)^-\leq|u_{\infty,n}|+|u_{1,n}|\quad \forall n\in\na.$$
Analogously, we obtain for arbitrary $n\in\na$
\begin{align*}
|u_{\infty,n}|=(u_{\infty,n})^++(u_{\infty,n})^-\leq |u|+|u_{\infty,1}|.
\end{align*}
\end{proof}

\begin{lem}\label{lem_L1convergence_umn}
For fixed $n\in\na$ there holds
\begin{align*}
\umn\rightarrow u_{\infty,n}\text{ in }L^1(Q_T)\text{ for }m\rightarrow\infty.
\end{align*}
Furthermore,
\begin{align*}
u_{\infty,n}\rightarrow u\text{ in }L^1(Q_T)\text{ for }n\rightarrow\infty.
\end{align*}
\end{lem}

\begin{proof}
Since $(\umn)_{m\in\na}$ converges for any $n\in\na$ to $u_{\infty,n}$ a.e. in $Q_T$ by Lemma \ref{lem_a.e.convergence_umn} and $|\umn|\leq g^n\in L^1(Q_T)$, we obtain by Lebesgue's dominated convergence theorem the $L^1$-convergence. Using analogously the convergence of $(u_{\infty,n})_{n\in\na}$ a.e. in $Q_T$ to $u$ from Lemma~\ref{lem_a.e.convergence_umn} and the boundedness $|u_{\infty,n}|\leq g\in L^1(Q_T) $ by Lemma \ref{lem_bound_b(vmn)}, we obtain $u_{\infty,n}\rightarrow u$ in $L^1(Q_T)$ for $n\rightarrow\infty$.
\end{proof}

\section{Passage to the limit}\label{sec_passage_to_the_limit}

Let (H$k$), (H$A$), (H$\varphi$), and (Hd) be satisfied and $\umn$ for $m,n\in\na$ the weak solution to $P(\uzmn,\fmn)$ as defined in Section \ref{section_approximation}, such that
\begin{align}\label{regularity_assumption_mn}
k\ast(\umn-\uzmn)\in {}_0W^{1,1}(0,T;L^1(\Omega))
\end{align}
is satisfied.
By Lemma \ref{lem_weak_entropy_solution} we know that $\vmn:=\varphi(\umn)$ is an entropy solution to $P(\uzmn,\fmn)$. Since $\varphi$ is continuous, we know by Lemma \ref{lem_a.e.convergence_umn} that
$$\vmn=\varphi(\umn)\overset{m\rightarrow\infty}\longrightarrow\varphi(u_{\infty,n})=:v_{\infty,n}\overset{n\rightarrow\infty}\longrightarrow\varphi(u)=:v \quad\text{a.e. in }Q_T.$$
Analogously, we get the convergences
\begin{align}\label{a.e.convergence_T_K(vmn)}
T_K(\vmn)\overset{m\rightarrow\infty}\longrightarrow T_K(v_{\infty,n})\overset{n\rightarrow\infty}\longrightarrow T_K(v)\text{ a.e. in }Q_T, \ \forall K>0.
\end{align}

\begin{lem}\label{lem_convergence_T_K(vmn)}
For all $K>0$ 
\begin{align*}
T_K(\vmn)\rightharpoonup T_K(v_{\infty,n})\text{ in }L^2(0,T;H^1_0(\Omega)),\forall n\in\na, \text{ for }m\rightarrow\infty
\end{align*}
and
\begin{align*}
T_K(v_{\infty,n})\rightharpoonup T_K(v)\text{ in }L^2(0,T;H^1_0(\Omega))\text{ for }n\rightarrow\infty.
\end{align*}
\end{lem}

\begin{proof}
We fix $K>0$. Obviously, we have
$$\|T_K(\vmn)\|_{L^2(Q_T)}^2\leq T|\Omega|K^2\quad\forall m,n\in\na.$$
Hence, we know by \eqref{a.e.convergence_T_K(vmn)} that there exist (not relabelled) subsequences of $(T_K(\vmn))_{m\in\na}$ and $(T_K(v_{\infty,n}))_{n\in\na}$ such that
\begin{align}\label{T_K(vmn)_convergenceL2}
\begin{aligned}
T_K(\vmn)&\rightharpoonup T_K(v_{\infty,n})\text{ in }L^2(0,T;L^2(\Omega))\text{ for }m\rightarrow\infty.\\
\text{and}\quad T_K(v_{\infty,n})&\rightharpoonup T_K(v)\text{ in }L^2(0,T;L^2(\Omega)\text{ for }n\rightarrow\infty.
\end{aligned}
\end{align}
Now we fix $m,n\in\na$ and use $T_K(\vmn)$ as a test function in \eqref{eq_weaksolution} to get
\begin{align*}
&\intq T_K(\vmn)\partial_t[k\ast(b(\vmn)-b(\vzmn))]+\intq(A\nabla\vmn,\nabla T_K(\vmn))\\
&=\intq \fmn T_K(\vmn).
\end{align*}
For $\lambda>0$ let $k_\lambda$ be the kernel associated to the Yosida approximation of the operator
\begin{align*}
D(L)&:=\left\{w\in L^1(0,T;L^{1}(\Omega)):k\ast w\in{}_0W^{1,1}(0,T;L^{1}(\Omega)) \right\},\\
L&:=\partial_t(k\ast\cdot).
\end{align*}
By using (H$A$) we obtain
\begin{align}\label{eq_1}
\begin{split}
&\intq T_K(\vmn)\partial_t[k_\lambda\ast(b(\vmn)-b(\vzmn))]+\nu\intq|\nabla T_K(\vmn)|^2\\
&\leq \intq \fmn T_K(\vmn)+\intq T_K(\vmn)\partial_t[(k_\lambda-k)\ast(b(\vmn)-b(\vzmn))].
\end{split}
\end{align}
Since $b(\vmn)-b(\vzmn)\in D(L)$ the last term converges to zero for $\lambda\rightarrow0$. The fundamental identity provides
\begin{align*}
\intq T_K(\vmn)\partial_t[k_\lambda\ast(b(\vmn)-b(\vzmn))]&\geq\intq\partial_t\left[k_\lambda\ast\int_{\vzmn}^{\vmn} T_K(\sigma)\,db(\sigma)\right]\\
&=\int_\Omega\left[k_\lambda\ast\int_{\vzmn}^{\vmn} T_K(\sigma)\,db(\sigma)\right](T).
\end{align*}
Letting $\lambda\rightarrow0$ in \eqref{eq_1} we obtain, since $k_\lambda\rightarrow k$ in $L^1(0,T)$,
\begin{align*}
\int_\Omega\left[k\ast\int_{\vzmn}^{\vmn}T_K(\sigma)\,db(\sigma)\right](T)+\nu\intq|\nabla T_K(\vmn)|^2\leq\intq \fmn T_K(\vmn).
\end{align*}
Since $k$ is non-negative and $b$ non-decreasing, we know, that 
\begin{align*}
\int_\Omega\left[k\ast\int_{\vzmn}^{\vmn}T_K(\sigma)\,db(\sigma)\right](T)\geq0.
\end{align*}
Using $|\fmn|\leq|f|$ and $\nu>0$ we get
$$\|\nabla T_K(\vmn)\|_{L^2(Q_T)}^2\leq \frac{K}{\nu}\|f\|_{L^1(Q_T)}.$$
Together with \eqref{a.e.convergence_T_K(vmn)} and \eqref{T_K(vmn)_convergenceL2} we conclude that $T_K(\vmn)\rightharpoonup T_K(v_{\infty,n})$ in $L^2(0,T;H^1_0(\Omega))$ for $m\rightarrow\infty$ and $T_K(v_{\infty,n})\rightharpoonup T_K(v)$ in $L^2(0,T;H^1_0(\Omega))$ for $n\rightarrow\infty$.
\end{proof}

\begin{teo}
Let (H$k$),(H$A$),(H$\varphi$), and (Hd) be satisfied. For any $m,n\in\na$ let $\umn$ be a weak solution to $P(\umn,\fmn)$ such that \eqref{regularity_assumption_mn} holds. Then $v:=\lim_{m,n\rightarrow\infty}\varphi(\umn)$ is an entropy solution to \ref{equation}.
\end{teo}

\begin{proof}
Let $m,n\in\na$. By Lemma \ref{lem_weak_entropy_solution} we know that $\vmn=\varphi(\umn)$ is an entropy solution to $P(\uzmn,\fmn)$. So, for any $S\in\mathcal{P},\phi\in H^1_0(\Omega)\cap L^\infty(\Omega)$ and $\zeta\in\mathcal{D}([0,T)),\zeta\geq0,$ there holds
\begin{align}\label{entropy_inequality_forLimes}
\begin{split}
&-\intq\zeta_t\left[k_1\ast\int_{\vzmn}^{\vmn}S(\sigma-\phi)\,db(\sigma)\right]+\intq\zeta\partial_t[k_2\ast(b(\vmn)-b(\vzmn))]S(\vmn-\phi)\\
&+\intq\zeta (A\nabla\vmn,\nabla S(\vmn-\phi))\leq \intq\zeta\fmn S(\vmn-\phi),
\end{split}
\end{align}
where $k_1,k_2\in L^1(0,T)$ are non-increasing and non-negative with $k=k_1+k_2$ and $k_2(0^+)<\infty$.\\
Since $\supp S'$ is compact, there exists a constant $L\geq0$ such that $\supp S'\subseteq[-L,L]$ and, therefore, for $M:=L+\|\phi\|_{L^\infty(\Omega)}$ we obtain
\begin{align*}
&\intq\zeta(A\nabla\vmn,\nabla S(\vmn-\phi))\\
&=\intq\zeta S'(T_M(\vmn)-\phi)(A\nabla T_M(\vmn),\nabla(T_M(\vmn)-\phi)).
\end{align*}
Indeed, if $|\vmn|\geq M$, we have
\begin{align*}
|\vmn-\phi|\geq|\vmn|-|\phi|\geq L+\|\phi\|_{L^\infty(\Omega)}-|\phi|\geq L
\end{align*}
and, therefore, $S'(\vmn-\phi)=0$ for $|\vmn|\geq M$. Since $S'$ is continuous, by \eqref{a.e.convergence_T_K(vmn)}, we get
$$S'(T_M(\vmn)-\phi)\overset{m\rightarrow\infty}\longrightarrow S'(v_{\infty,n}-\phi)\overset{n\rightarrow\infty}\longrightarrow S'(T_M(v)-\phi)\text{ a.e. in } Q_T.$$
Additionally, $S'$ has compact support and, therefore, by using Lebesgue's dominated convergence theorem, we obtain
\begin{align}\label{convergence S'(T_M(vmn)-phi)}
\begin{aligned}
S'(T_M(\vmn)-\phi)&\rightarrow S'(T_M(v_{\infty,n})-\phi)\text{ in }L^1(Q_T)\text{ for }m\rightarrow\infty\\
\text{and}\quad S'(T_M(v_{\infty,n})-\phi)&\rightarrow S'(T_M(v)-\phi)\text{ in }L^1(Q_T)\text{ for }n\rightarrow\infty
\end{aligned}
\end{align}
Now, we consider
\begin{align*}
&\zeta \Big(A\nabla T_M(\vmn),S'(T_M(\vmn)-\phi)\nabla T_M(\vmn)\Big)\\
&=\zeta S'(T_M(\vmn)-\phi)\Big(A\nabla (T_M(\vmn)-T_M(v)),\nabla(T_M(\vmn)-T_M(v))\Big)\\
&\ +\zeta \Big(A\nabla T_M(v),S'(T_M(\vmn)-\phi)\nabla(T_M(\vmn)-T_M(v))\Big)\\
&\ +\zeta \Big(A\nabla T_M(\vmn),S'(T_M(\vmn)-\phi)\nabla T_M(v)\Big)\\
&=:I_1+I_2+I_3.
\end{align*}
Having in mind that $\zeta\geq0$ and $S'\geq0$, the coercivity condition (H$A$) implies
\begin{align*}
\intq I_1\geq\nu\intq\zeta|\nabla T_M(\vmn)|^2 S'(T_M(\vmn)-\phi)\geq0\quad\forall m,n\in\na
\end{align*}
Using Lemma \ref{lem_convergence_T_K(vmn)} and the convergence \eqref{convergence S'(T_M(vmn)-phi)}, we know
\begin{align*}
\lim_{n\rightarrow\infty}\lim_{m\rightarrow\infty}\intq I_2=0
\end{align*}
and
\begin{align*}
\lim_{n\rightarrow\infty}\lim_{m\rightarrow\infty}\intq I_3= &\intq \zeta (A\nabla T_M(v),S'(T_M(v)-\phi)\nabla T_M(v))\\
&=\intq\zeta (A\nabla v,S'(v-\phi)\nabla v).
\end{align*}
It follows that
\begin{align*}
&\liminf_{n\rightarrow\infty}\liminf_{m\rightarrow\infty}\intq\zeta \big(A\nabla T_M(\vmn),S'(T_M(\vmn)-\phi)\nabla (T_M(\vmn)-\phi)\big)\\
&\geq \intq\zeta (A\nabla v,\nabla S(v-\phi)).
\end{align*}
According to the first term in \eqref{entropy_inequality_forLimes} we know that 
\begin{align*}
\int_{\vzmn}^{\vmn}S(\sigma-\phi)\,db(\sigma)\overset{m\rightarrow\infty}\longrightarrow\int_{v_0^{\infty,n}}^{v_{\infty,n}}S(\sigma-\phi)\,db(\sigma)\overset{n\rightarrow\infty}\longrightarrow\int_{v_0}^vS(\sigma-\phi)\,db(\sigma)\text{ a.e. in }Q_T,
\end{align*}
where $v_{\infty,n}^0:=\varphi(u_{\infty,n}^0):=\varphi(\lim_{m\rightarrow\infty}\uzmn)$.
Using Lemma \ref{lem_bound_b(vmn)}, since $S$ is bounded and $|b(\vzmn)|\leq|b(v_0)|$, we know that there exists a constant $C\geq0$ such that
\begin{align*}
\left|\int_{\vzmn}^{\vmn}S(\sigma-\phi)\,db(\sigma)\right|\leq C(g^n+|b(v_0|)\in L^1(Q_T)\quad\forall m,n\in\na.
\end{align*}
Analogously there exists a constant $C\geq0$ such that
\begin{align*}
\left|\int_{v_{0}^{\infty,n}}^{v_{\infty,n}}S(\sigma-\phi)\,db(\sigma)\right|\leq C(g+|b(v_0)|)\in L^1(Q_T)\quad\forall n\in\na.
\end{align*}
Hence Lebesgue's dominated convergence theorem implies the convergence in $L^1(Q_T)$ and, therefore,
\begin{align*}
\lim_{n\rightarrow\infty}\lim_{m\rightarrow\infty}\left(-\intq\zeta_t\left[k_1\ast\int_{\vzmn}^{\vmn}S(\sigma-\phi)\,db(\sigma)\right]\right)=-\intq\zeta_t\left[k_1\ast\int_{v_0}^vS(\sigma-\phi)\,db(\sigma)\right].
\end{align*}
It remains to show the convergence for 
\begin{align*}
\intq\zeta\partial_t[k_2\ast(b(\vmn)-b(\vzmn))]S(\vmn-\phi).
\end{align*}
Using Lemma \ref{lem_L1convergence_umn}, there holds for a.e. $t\in(0,T)$
\begin{align*}
&\partial_t[k_2\ast(b(\vmn)-b(\vzmn))](t)\\
&=k_2(0^+)(b(\vmn(t)-b(\vzmn))+\int_0^t b(\vmn(t-s))-b(\vzmn)\,dk_2(s)\\
&\overset{m\rightarrow\infty}\longrightarrow k_2(0^+)(b(v_{\infty,n}(t))-b(v_{\infty,n}^0))+\int_0^t b(v_{\infty,n}(t-s))-b(v_{\infty,n}^0)\,dk_2(s)\\
&\overset{n\rightarrow\infty}\longrightarrow k_2(0^+)(b(v(t))-b(v_0))+\int_0^t b(v(t-s))-b(v_0)\,dk_2(s),
\end{align*}
where the convergences hold in $L^1(Q_T)$. Consequently,
\begin{align*}
&\lim_{n\rightarrow\infty}\lim_{m\rightarrow\infty}\intq\zeta\partial_t[k_2\ast(b(\vmn)-b(\vzmn))]S(\vmn-\phi)\\
&=\intq\zeta\partial_t[k_2\ast(b(v)-b(v_0))]S(v-\phi).
\end{align*}
Summing up, we get
\begin{align*}
&-\intq\zeta_t\left[k_1\ast\int_{v_0}^v S(\sigma-\phi)\,db(\sigma)\right]+\intq\zeta\partial_t[k_2\ast(b(v)-b(v_0))]S(v-\phi)\\
&+\intq\zeta(A\nabla v,\nabla S(v-\phi))\\
&\leq\lim_{n\rightarrow\infty}\lim_{m\rightarrow\infty}\intq\zeta \fmn S(\vmn-\phi)=\intq\zeta f S(v-\phi)
\end{align*}
and hence $v$ is an entropy solution to \ref{equation}.
\end{proof}

\begin{remark}
Let for $i=1,2$, $u_{0,i}\in L^1(\Omega),\ f_i\in L^1(Q_T)$, and $v_i$ be an entropy solution to $P(u_{0,i},f_i)$, such that $v_i$ is the limit of $\varphi(u_{m,n}^i)$, where $u_{m,n}^i$ is a weak solution to $P(u_{0,i}^{m,n},f_i^{m,n})$. Here, $u_{0,i}^{m,n}$ and $f_i^{m,n}$ are the bounded approximations of $u_{0,i}$ and $f_i$ defined analogously as in Section \ref{section_approximation}.
Then the contraction principle
\begin{align*}
\|b(v_1)-b(v_2)\|_{L^1(Q_T}\leq T\|u_{0,1}-u_{0,2}\|_{L^1(\Omega)}+\|l\|_{L^1(0,T)}\|f_1-f_2\|_{L^1(Q_T)}
\end{align*}
holds. The proof is a consequence of the convergence of the approximate solutions.
\end{remark}

%\nocite{*}
\bibliographystyle{abbrv}

\bibliography{L1_Bibliography.bib}

\begin{thebibliography}{10}

\bibitem{A19}
G.~Akagi.
\newblock Fractional flows driven by subdifferentials in {H}ilbert spaces.
\newblock {\em Israel J. Math.}, 234:809--862, 2019.

\bibitem{ACV17}
M.~Allen, L.~Caffarelli, and A.~Vasseur.
\newblock Porous medium flow with both a fractional potential pressure and
  fractional time derivative.
\newblock {\em Chin. Ann. Math. Ser. B}, 38:45--82, 2017.

\bibitem{C99}
M.~Caputo.
\newblock Diffusion of fluids in porous media with memory.
\newblock {\em Geothermics}, 28:113--130, 1999.

\bibitem{DVV19}
S.~Dipierro, E.~Valdinoci, and V.~Vespri.
\newblock Decay estimates for evolutionary equations with fractional
  time-diffusion.
\newblock {\em J. Evol. Equ.}, 19:435--462, 2019.

\bibitem{DNA19}
J.-D. Djida, J.~J. Nieto, and I.~Area.
\newblock Nonlocal time porous medium equation with fractional time derivative.
\newblock {\em Rev. Mat. Complut.}, 32:273--304, 2019.

\bibitem{G85}
G.~Gripenberg.
\newblock Volterra integro-differential equations with accretive nonlinearity.
\newblock {\em Journal of Differential Equations}, 60:57--79, 1985.

\bibitem{JDiss}
V.~G. Jakubowski.
\newblock {\em Nonlinear elliptic-parabolic integro-differential equations with
  $L_1$-data: existence, uniqueness, asymptotics}.
\newblock PhD thesis, University of Duisburg-Essen, 2001.

\bibitem{JW04}
V.~G. Jakubowski and P.~Wittbold.
\newblock On a nonlinear elliptic-parabolic integro-differential equation with
  {$L_1$}-data.
\newblock {\em J. Differential Equations}, 197:427--445, 2004.

\bibitem{KSVZ16}
J.~Kemppainen, J.~Siljander, V.~Vergara, and R.~Zacher.
\newblock Decay estimates for time-fractional and other nonlocal in time
  subdiffusion equations in $\mathbb{R}^d$.
\newblock {\em Math. Ann.}, 366:941--979, 2016.

\bibitem{K08}
A.~N. Kochubei.
\newblock Distributed order calculus and equations of ultraslow diffusion.
\newblock {\em J. Math. Anal. Appl.}, 340:252--281, 2008.

\bibitem{K11}
A.~N. Kochubei.
\newblock General fractional calculus, evolution equations, and renewal
  processes.
\newblock {\em Integr. Equ. Oper. Theory}, 71:583--600, 2011.

\bibitem{KR18}
A.~Kubica and K.~Ryszewska.
\newblock Decay of solutions to parabolic-type problem with distributed order
  {C}aputo derivative.
\newblock {\em J. Math. Anal. Appl.}, 465:75--99, 2018.

\bibitem{LRS19}
W.~Liu, M.~R\"{o}ckner, and J.~L. {da Silva}.
\newblock Strong dissipativity of generalized time-fractional derivatives and
  quasi-linear (stochastic) partial differential equations.
\newblock {\em Journal of Functional Analysis}, 281(8):109135, 2021.

\bibitem{MK00}
R.~Metzler and J.~Klafter.
\newblock The random walk's guide to anomalous diffusion: a fractional dynamics
  approach.
\newblock {\em Phys. Rep.}, 339:1--77, 2000.

\bibitem{MK04}
R.~Metzler and J.~Klafter.
\newblock The restaurant at the end of the random walk: recent developements in
  the description of anomalous transport by fractional dynamics.
\newblock {\em J. Phys. A}, 37:R161--R208, 2004.

\bibitem{N71}
J.~W. Nunziato.
\newblock On heat conduction in materials with memory.
\newblock {\em Quart. Appl. Math.}, 29:187--204, 1971.

\bibitem{P15}
{\L}.~P{\l}ociniczak.
\newblock {Analytical studies of a time-fractional porous medium equation.
  Derivation, approximation and applications}.
\newblock {\em Commun. Nonlinear Sci. Numer. Simul.}, 24:169--183, 2015.

\bibitem{Pruess}
J.~Prüss.
\newblock {\em Evolutionary Integral Equations and Applications}.
\newblock Monographs in mathematics 87. Birkhäuser, Basel, 1993.

\bibitem{S20}
N.~Sapountzoglou.
\newblock Entropy solutions to doubly nonlinear integro-differential equations.
\newblock {\em Nonlinear Anal.}, 192:111656, 2020.

\bibitem{SW18}
M.~Scholtes and P.~Wittbold.
\newblock Existence of entropy solutions to a doubly nonlinear
  integro-differential equation.
\newblock {\em Differential Integral Equations}, 31:465--496, 2018.

\bibitem{VZ15}
V.~Vergara and R.~Zacher.
\newblock Optimal decay estimates for time-fractional and other non-local
  subdiffusion equations via energy methods.
\newblock {\em SIAM J. Math. Anal.}, 47:210--239, 2015.

\bibitem{VZ17}
V.~Vergara and R.~Zacher.
\newblock Stability, integrability, and blowup for time fractional and other
  nonlocal in time semilinear subdiffusion equations.
\newblock {\em J. Evol. Equ.}, 17:599--626, 2017.

\bibitem{WWZ20}
P.~Wittbold, P.~Wolejko, and R.~Zacher.
\newblock Bounded weak solutions of time-fractional porous medium type and more
  general nonlinear and degenerate evolutionary integro-differential equations.
\newblock {\em J. Math. Anal. Appl.}, 499(125007):20pp., 2021.

\bibitem{Z19}
R.~Zacher.
\newblock Time fractional diffusion equations: solution concepts, regularity
  and long-time behaviour.
\newblock In {\em Handbook of fractional calculus without applications},
  volume~2, pages 159--179. De Gruyter, Berlin, 2019.

\end{thebibliography}

\end{document}